\documentclass{amsart}

\usepackage{amsmath}
\usepackage{amsthm}
\usepackage{amsfonts}

\newcommand{\reals}{\mathbb{R}}

\newcommand{\bracketb}[1]{\Big[#1\Big]}
\newcommand{\bracketc}[1]{\bigg[#1\bigg]}

\newcommand{\pb}[1]{\left\{#1\right\}}

\newcommand{\norm}[1]{\left|\left|#1\right|\right|}

\newcommand{\paraa}[1]{\big(#1\big)}
\newcommand{\parab}[1]{\Big(#1\Big)}
\newcommand{\parac}[1]{\bigg(#1\bigg)}

\newtheorem{theorem}{Theorem}[section]

\newtheorem{lemma}[theorem]{Lemma}
\newtheorem{proposition}[theorem]{Proposition}

\theoremstyle{definition}

\theoremstyle{remark}

\numberwithin{equation}{section}

\newcommand{\tr}{\operatorname{tr}}
\newcommand{\Tr}{\operatorname{Tr}}
\newcommand{\Z}{\mathcal{Z}}
\newcommand{\A}{\mathcal{A}}
\newcommand{\B}{\mathcal{B}}
\renewcommand{\P}{\mathcal{P}}
\renewcommand{\S}{\mathcal{S}}
\newcommand{\J}{\mathcal{J}}
\newcommand{\JM}{\mathcal{J}_M}
\newcommand{\SA}{\S_A}

\newcommand{\STA}{\SA^T}

\newcommand{\gb}{\,\bar{\!g}}
\renewcommand{\d}{\partial}
\newcommand{\TSigma}{T\Sigma}
\newcommand{\eps}{\varepsilon}

\newcommand{\nablab}{\bar{\nabla}}

\newcommand{\Gammab}{\bar{\Gamma}}
\newcommand{\Rb}{\bar{R}}
\newcommand{\TN}{T^{(N)}}

\newcommand{\Nh}{\hat{N}}
\newcommand{\Yt}{\tilde{Y}}

\title[Classical geometry of embedded surfaces and their Poisson brackets]
{On the classical geometry of embedded surfaces in terms of Poisson brackets}

\author{Joakim Arnlind}
\address{Max Planck Institute for Gravitational Physics (AEI), Am M\"uhlenberg 1, D-14476 Golm, Germany.}
\email{joakim.arnlind@aei.mpg.de}

\author{Jens Hoppe}
\address{Department of Mathematics, KTH, S-10044 Stockholm, Sweden.}
\email{hoppe@math.kth.se}

\author{Gerhard Huisken}
\address{Max Planck Institute for Gravitational Physics (AEI), Am M\"uhlenberg 1, D-14476 Golm, Germany.}
\email{gerhard.huisken@aei.mpg.de}

\thanks{}

\subjclass[2000]{}

\begin{document}

\begin{abstract}
  We consider surfaces embedded in a Riemannian manifold of arbitrary
  dimension and prove that many aspects of their differential geometry
  can be expressed in terms of a Poisson algebraic structure on the
  space of smooth functions of the surface. In particular, we find
  algebraic formulas for Weingarten's equations, the
  complex structure and the Gaussian curvature.
\end{abstract}

\maketitle

\section{Introduction}

\noindent Given a manifold $\Sigma$, it is interesting to study in
what ways information about the geometry of $\Sigma$ can be
extracted as algebraic properties of the algebra of smooth functions
$C^\infty(\Sigma)$. In case $\Sigma$ is a Poisson manifold, this
algebra has a second (apart from the commutative multiplication of
functions) bilinear (non-associative) algebra structure realized as the
Poisson bracket. The bracket is compatible with the
commutative multiplication via Leibniz rule, thus carrying the basic
properties of a derivation. 

On a surface $\Sigma$, with local coordinates $u^1$ and $u^2$, one may
define
\begin{align*}
  \{f,h\} = \frac{1}{\sqrt{g}}\parac{\frac{\d f}{\d u^1}\frac{\d h}{\d u^2}-
  \frac{\d h}{\d u^1}\frac{\d f}{\d u^2}},
\end{align*}
where $g$ is the determinant of the metric tensor, and one can readily
check that $\paraa{C^\infty(\Sigma),\{\cdot,\cdot\}}$ is a Poisson
algebra. Having only this very particular combination of derivatives
at hand, it seems at first unlikely that one can encode geometric
information of $\Sigma$ in Poisson algebraic
expressions. Surprisingly, it turns out that many differential
geometric quantities can be computed in a completely algebraic way, as
stated in Theorem \ref{thm:gaussiancurvature} and in Theorem
\ref{thm:complexstructure}. For instance, the Gaussian curvature of a
surface embedded in $\reals^3$ can be written as
\begin{align*}
  K = -\frac{1}{2}\sum_{i,j=1}^3\{x^i,n^j\}\{x^j,n^i\},
\end{align*}
where $x^i(u^1,u^2)$ are the embedding coordinates and $n^i(u^1,u^2)$
are the components of a unit normal vector at each point of $\Sigma$.

Let us also mention that our initial motivation for studying this
problem came from matrix regularizations of Membrane Theory. Classical
solutions in Membrane Theory are 3-manifolds with vanishing mean
curvature in $\reals^{1,d}$. Considering one of the coordinates to be
time, the problem can also be formulated in a dynamical way as
surfaces sweeping out volumes of vanishing mean curvature. In this
context, a regularization was introduced replacing the infinite
dimensional function algebra on the surface by an algebra of $N\times
N$ matrices \cite{h:phdthesis}. If we let $\TN$ be a linear map from
smooth functions to hermitian $N\times N$ matrices, the regularization
is required to fulfill
\begin{align*}
  &\lim_{N\to\infty}\norm{\TN(f)\TN(g)-\TN(fg)}=0,\\
  &\lim_{N\to\infty}\norm{N[\TN(f),\TN(h)]-i\TN(\{f,h\})}=0,
\end{align*}
where $||\cdot||$ denotes the operator norm, and therefore it is
natural to regularize the system by replacing (commutative)
multiplication of functions by (non-commutative) multiplication of
matrices and Poisson brackets of functions by commutators of matrices.

Although we may very well consider $\TN(\frac{\d f}{\d
  u^1})$, its relation to $\TN(f)$ is in general not simple. However, the
particular combination of derivatives found in $\TN(\{f,h\})$ is
easily expressed in terms of a commutator of $\TN(f)$ and $\TN(h)$. In
the context of Membrane Theory, it is desirable to have geometrical
quantities in a form that can easily be regularized, which is the case
for any expression constructed out of multiplications and Poisson
brackets.

\section{Preliminaries}

\noindent To introduce the relevant notations, we shall recall some
basic facts about submanifolds, in particular Gauss' and Weingarten's
equations (see
e.g. \cite{kn:foundationsDiffGeometryI,kn:foundationsDiffGeometryII}
for details). Let $\Sigma$ be a two dimensional manifold embedded in a
Riemannian manifold $M$ with $\dim M=2+p\equiv m$. Local coordinates
on $M$ will be denoted by $x^1,\ldots,x^m$, local coordinates on
$\Sigma$ by $u^1,u^2$, and we regard $x^1,\ldots,x^m$ as being
functions of $u^1,u^2$ providing the embedding of $\Sigma$ in $M$. The
metric tensor on $M$ is denoted by $\gb_{ij}$ and the induced metric
on $\Sigma$ by $g_{ab}$; indices $i,j,k,l,n$ run from $1$ to $m$,
indices $a,b,c,d,p,q$ run from $1$ to $2$ and indices $A,B,C,D$ run
from $1$ to $p$. Furthermore, the covariant derivative and the
Christoffel symbols in $M$ will be denoted by $\nablab$ and
$\Gammab^{i}_{jk}$ respectively.

The tangent space $\TSigma$ is regarded as a subspace of the tangent
space $TM$ and at each point of $\Sigma$ one can choose
$e_a=(\d_ax^i)\d_i$ as basis vectors in $\TSigma$, and in this basis
we define $g_{ab}=\gb(e_a,e_b)$. Moreover, we choose a set of normal
vectors $N_A$, for $A=1,\ldots,p$, such that
$\gb(N_A,N_B)=\delta_{AB}$ and $\gb(N_A,e_a)=0$.

The formulas of Gauss and Weingarten express the relation between the
covariant derivative in $M$ and the covariant derivative in $\Sigma$
(denoted by $\nabla$) as
\begin{align}
  &\nablab_X Y = \nabla_X Y + \alpha(X,Y)\label{eq:GaussFormula}\\
  &\nablab_XN_A = -W_A(X) + D_XN_A\label{eq:WeingartenFormula}
\end{align}
where $X,Y\in \TSigma$ and $\nabla_X Y$, $W_A(X)\in\TSigma$ and
$\alpha(X,Y)$, $D_XN_A\in\TSigma^\perp$. By expanding $\alpha(X,Y)$ in
the basis $\{N_1,\ldots,N_p\}$ one can write (\ref{eq:GaussFormula}) as
\begin{align}
  &\nablab_X Y = \nabla_X Y + \sum_{A=1}^ph_A(X,Y)N_A,\label{eq:GaussFormulah}
\end{align}
and we set $h_{A,ab} = h_A(e_a,e_b)$. It follows from the above equations that following relations exist
\begin{align}
  h_{A,ab} &= -\gb\paraa{e_a,\nablab_b N_A}\\
  (W_A)^a_b &= g^{ac}h_{A,cb},
\end{align}
where $g^{ab}$ denotes the inverse of $g_{ab}$. From Gauss' formula
(\ref{eq:GaussFormula}) one obtains an expression for the
curvature $R$ of $\Sigma$ in terms of the curvature $\Rb$ of $M$
\begin{equation}\label{eq:GaussEquation}
  \begin{split}
    g\paraa{R(X,Y)Z,W} =
    \gb&\paraa{\Rb(X,Y)Z,W}-\gb\paraa{\alpha(X,Z),\alpha(Y,W)}\\
    &+\gb\paraa{\alpha(Y,Z),\alpha(X,W)},    
  \end{split}
\end{equation}
where $X,Y,Z,W\in\TSigma$. The Gaussian curvature $K$ of $\Sigma$ can
be computed as the sectional curvature (of $\Sigma$) in the plane
spanned by $e_1$ and $e_2$, i.e.
\begin{align}
  K = \frac{g\paraa{R(e_1,e_2)e_2,e_1}}{g(e_1,e_1)g(e_2,e_2)-g(e_1,e_2)^2}
\end{align}
which, by using (\ref{eq:GaussEquation}), yields
\begin{align}\label{eq:GaussianCurvature}
  K = \frac{1}{g}\gb\paraa{\Rb(e_1,e_2)e_2,e_1}+
  \sum_{A=1}^p\frac{\det(h_{A,ab})}{g}
\end{align}
where $g=\det(g_{ab})$. We also recall the mean curvature vector, defined as
\begin{align}
  H = \frac{1}{2}\sum_{A=1}^p\paraa{\tr W_A}N_A.
\end{align}

\section{Poisson algebraic formulation}

\noindent In this section we will prove that one can express many
aspects of the differential geometry of an embedded surface in terms
of a Poisson algebra structure introduced on
$C^\infty(\Sigma)$. Namely, let $\rho:\Sigma\to\reals$ be an arbitrary non-vanishing
density, i.e. an object transforming as
\begin{align}
  \tilde{\rho}(v) = \left|\frac{\d(u^1,u^2)}{\d(v^1,v^2)}\right|\rho\paraa{u(v)}
\end{align}
under general coordinate transformations, and define
\begin{align}\label{eq:PbracketDef}
  \{f,h\} = \frac{1}{\rho}\eps^{ab}\paraa{\d_af}\paraa{\d_b h}
\end{align}
for all $f,h\in C^\infty(\Sigma)$, where $\eps^{ab}$ is antisymmetric
with $\eps^{12}=1$. It is straightforward to check that
$\paraa{C^\infty(\Sigma),\{\cdot,\cdot\}}$ is a Poisson algebra. 

The above Poisson bracket also arises from the choice of a
volume form on $\Sigma$. Namely, since any 2-form is closed on a two dimensional manifold,
a volume form $\omega$ is also a symplectic form. For any smooth
function $f$, a symplectic form defines a vector field
$X_f\in\TSigma$ associated with $f$ through the relation
\begin{equation*}
  \omega(X_f,Y) = df(Y)
\end{equation*}
for all $Y\in\TSigma$. Furthermore, since $\omega$ is closed, one
defines a Poisson bracket by setting
\begin{equation*}
  \{f,h\} = \omega(X_f,X_h),
\end{equation*}
which, in local coordinates where $\omega=\rho(u^1,u^2)du^1\wedge du^2$,
coincides with (\ref{eq:PbracketDef}).

Let $x^1(u^1,u^2),\ldots,x^m(u^1,u^2)$ be the embedding of $\Sigma$,
and let $n_A^1,\ldots,n_A^m$ denote the components of the normal
vector $N_A$. We introduce the following tensors
\begin{align}
  &\P^{ij} = \{x^i,x^j\}\\
  &\SA^{ij} = \frac{1}{\rho}\eps^{ab}\paraa{\d_ax^i}\paraa{\nablab_b N_A}^j= \{x^i,n_A^j\}+\{x^i,x^k\}\Gammab^j_{kl}n_A^l,
\end{align}
and by lowering one of the indices one can regard $\P$ and $\SA$ as maps
$TM\to TM$. Our convention is to lower the second index,
i.e
\begin{align*}
  &\P(X) = \P^{ik}\gb_{kj}X^j\d_i\\
  &\SA(X) = \SA^{ik}\gb_{kj}X^j\d_i\\
  &\S^T_A(X) = \gb_{ik}\SA^{kj}X^i\d_j.
\end{align*}
Out of these objects, we define two compound maps $TM\to TM$ as
\begin{align}
  &\A_A(X) = -\P\STA(X)\\
  &\B_A(X) = \P\SA(X),
\end{align}
whose components in the coordinate basis are
\begin{align}
  (\A_A)^i_k &= \pb{x^i,x^j}\gb_{jj'}\{n_A^{j'},x^{k'}\}\gb_{k'k}
  +\{x^i,x^j\}\gb_{jj'}\Gammab^{j'}_{ll'}n_A^{l'}\{x^l,x^{k'}\}\gb_{k'k}\\
  (\B_A)^i_k &= \{x^i,x^{j}\}\gb_{jj'}\{x^{j'},n_A^{k'}\}\gb_{k'k}
  +\{x^i,x^{j}\}\gb_{jj'}\{x^{j'},x^l\}\Gammab^{k'}_{ll'}n_A^{l'}\gb_{k'k}.
\end{align}

\noindent Let us now investigate some properties of the maps defined above.

\begin{proposition}\label{prop:SPproperties}
  For $X\in TM$ it holds that
  \begin{align*}
    &\SA(X) = -\frac{1}{\rho}\gb\paraa{X,\nablab_aN_A}\eps^{ab}e_b\\
    &\P(X) = -\frac{1}{\rho}\gb\paraa{X,e_a}\eps^{ab}e_b\\
    &\P^2(X) = -\frac{g}{\rho^2}\gb\paraa{X,e_a}g^{ab}e_b.
  \end{align*}
  In particular, $\P(X),\SA(X)\in\TSigma$, and for $Y\in \TSigma$ one obtains $\P^2(Y)=-(g/\rho^2)Y$.
\end{proposition}

\begin{proof}
  Let us provide a proof for the statements concerning $\P$. The
  statement about $\SA$ is proven in an analogous way.

  Any $X\in TM$ is written in coordinates as $X^i\d_i$. Applying $\P$
  to this vector yields
  \begin{align*}
    \P(X) &= \frac{1}{\rho}\eps^{ab}\paraa{\d_ax^i}\paraa{\d_bx^j}\gb_{jk}X^k\d_i
    = \frac{1}{\rho}\gb\paraa{X,e_b}\eps^{ab}\paraa{\d_ax^i}\d_i\\
    &= -\frac{1}{\rho}\gb\paraa{X,e_a}\eps^{ab}e_b.
  \end{align*}
  Applying $\P$ once more gives
  \begin{align*}
    \P^2(X) = -\frac{1}{\rho}\gb\paraa{\P(X),e_a}\eps^{ab}e_b
    = \frac{1}{\rho^2}\eps^{ab}\eps^{pq}g_{qa}\gb\paraa{X,e_p}e_b,
  \end{align*}
  and using the fact that $gg^{bp}=\eps^{ab}\eps^{qp}g_{qa}$ one obtains
  \begin{align*}
    \P^2(X) = -\frac{g}{\rho^2}\gb(X,e_p)g^{pb}e_b,
  \end{align*}
  which proves the statement.
\end{proof}

\noindent Since $\P$ and $\S_A$ can be restricted to $\TSigma$, one
may consider components both of the type $\P^i_j$ and of the type
$\P^a_b$. Therefore, there are two possible ways of defining a
trace, namely
\begin{align*}
  &\Tr\P = \P^i_i,\\
  &\tr\P = \P^a_a.
\end{align*}

\begin{proposition}
  With $W_A$ denoting the Weingarten maps it holds that
  \begin{align*}
    &\Tr\B_A = \tr\B_A = \Tr\A_A=\tr \A_A=\frac{g}{\rho^2}\tr W_A\\
    &\Tr\B_A^2 = \tr\B_A^2 = \Tr\A_A^2=\tr \A_A^2=\frac{g}{\rho^2}\tr W_A^2.
  \end{align*}
\end{proposition}

\noindent It turns out that the map $\B_A$, when restricted to
$\TSigma$, is actually proportional to the Weingarten map.

\begin{proposition}\label{prop:Bproperties}
  For $X\in TM$ it holds that
  \begin{align*}
    \B_A(X) = -\frac{g}{\rho^2}\gb\paraa{X,\nablab_aN_A}g^{ab}e_b,
  \end{align*}
  and, in particular, if $Y\in\TSigma$ then
  \begin{align*}
    \B_A(Y) = \frac{g}{\rho^2}W_A(Y).
  \end{align*}
\end{proposition}

\begin{proof}
  Using the results in Proposition \ref{prop:SPproperties} one computes
  \begin{align*}
    \B_A(X) &= \P\parab{-\frac{1}{\rho}\gb(X,\nablab_aN_A)\eps^{ab}e_b}
    = \frac{1}{\rho^2}\eps^{ab}\eps^{pq}g_{bp}\gb\paraa{X,\nablab_aN_A}e_q\\
    &= -\frac{g}{\rho^2}\gb\paraa{X,\nablab_aN_A}g^{aq}e_q.
  \end{align*}
  Let us take $Y\in\TSigma$ and write $Y=Y^ce_c$
  \begin{align*}
    \B_A(Y) &= -\frac{g}{\rho^2}Y^c\gb(e_c,\nablab_aN_A)g^{ab}e_b
    = \frac{g}{\rho^2}Y^ch_{A,ca}g^{ab}e_b\\
    &= \frac{g}{\rho^2}Y^c(W_A)_c^be_b = \frac{g}{\rho^2}W_A(Y).\qedhere
  \end{align*}
\end{proof}

\noindent The trace of the squares of $\P$ and $\SA$ also contain
geometric information, as shown in the next Proposition.

\begin{proposition}\label{prop:SPTraceSquare}
  For the maps $\P$ and $\S_A$ it holds that
  \begin{align}
    &\Tr\SA^2 = \tr\SA^2 = -\frac{2}{\rho^2}\det(h_{A,ab})\\
    &\Tr\P^2 = \tr\P^2 = -2\frac{g}{\rho^2}.
  \end{align}
\end{proposition}

\begin{proof}
  We provide proofs for the statements involving $\tr$. It is easy to
  see that the components of $\SA$, when restricted to $\TSigma$, are
  \begin{align*}
    (\SA)^a_b = -\frac{1}{\rho}\eps^{ac}h_{A,cb},
  \end{align*}
  which implies that
  \begin{align*}
    \tr\SA^2 = (\SA)^a_c(\SA)^c_a 
    = \frac{1}{\rho^2}\eps^{ap}h_{A,pc}\eps^{cq}h_{A,qa}
    = -\frac{2}{\rho^2}\det(h_{A}).
  \end{align*}
  Similarly, one finds that $\P^a_b = -g_{bc}\eps^{ca}/\rho$, which
  implies that
  \begin{align*}
    \tr\P^2 &= \P^a_b\P^b_a = \frac{1}{\rho^2}\eps^{ca}\eps^{pb}g_{bc}g_{ap}
    =-2\frac{g}{\rho^2}.\qedhere
  \end{align*}
\end{proof}

\noindent From (\ref{eq:GaussianCurvature}) it now follows that one
can compute the Gaussian curvature in terms of traces of $\SA^2$. We
collect this result, together with the expression for the mean
curvature vector, in the following theorem.
\begin{theorem}\label{thm:gaussiancurvature}
  Let $K$ denote the Gaussian curvature and let $H$ denote the mean
  curvature vector of $\Sigma$. Then
  \begin{align}
    &K = \frac{1}{g}\gb\paraa{\Rb(e_1,e_2)e_2,e_1}
    -\frac{\rho^2}{2g}\sum_{A=1}^p\Tr\SA^2,\\
    &H = \frac{\rho^2}{2g}\sum_{A=1}^p\paraa{\Tr\B_A}N_A.
  \end{align}
\end{theorem}

\noindent Note that when $M=\reals^m$ the above expressions become
\begin{align}
  &K = -\frac{\rho^2}{2g}\sum_{A=1}^p\sum_{i,j=1}^m\{x^i,n_A^j\}\{x^j,n_A^i\}\label{eq:gausscurvRm}\\
  &H =
  \frac{\rho^2}{2g}\sum_{A=1}^p\sum_{i,j,k=1}^m\{x^i,x^j\}\{x^j,n_A^i\}n_A^k\d_k.\label{eq:meancurvRm}
\end{align}

\noindent Coming back to Weingarten's formula for the covariant
derivative of a normal vector
\begin{align*}
  \nablab_XN_A = -W_A(X) + D_XN_A,
\end{align*}
we have shown that the Weingarten maps $W_A$ can be written in terms
of $\B_A$. One may ask the question if there is a relation also
between $D_XN_A$ and $\B_A$? Surprisingly, there is such a
relation. Namely, let
\begin{align}
  (D_X)_{AB} = \gb\paraa{N_A,D_XN_B}
\end{align}
denote the components of the covariant derivative (in the direction of
$X\in\TSigma$) in the normal space with respect to the basis
$N_1,\ldots,N_p$. Then one has the following result.
\begin{proposition}
  Let $D$ denote the covariant derivative in the normal space. Then
  for all $X\in\TSigma$ it holds that
  \begin{align*}
     \gb\paraa{\B_A(N_B),X} = \frac{g}{\rho^2}(D_X)_{AB},
  \end{align*}
  where $(D_X)_{AB}$ denotes the components of $D_X$ relative to the
  basis $N_1,\ldots,N_p$.
\end{proposition}

\begin{proof}
  For a vector $X=X^ae_a$, it follows from Weingarten's formula (\ref{eq:WeingartenFormula}) that
  \begin{align*}
    (D_X)_{BA} = \gb\paraa{N_B,\nablab_X N_A} = X^a\gb(N_B,\nablab_a N_A).
  \end{align*}
  On the other hand, with the formula from Proposition
  \ref{prop:Bproperties}, one computes
  \begin{align*}
     \gb\paraa{\B_A(N_B),X} &= -\frac{g}{\rho^2}\gb\paraa{N_B,\nablab_aN_A}g^{ab}g_{bc}X^c
     = -\frac{g}{\rho^2}\gb\paraa{N_B,\nablab_aN_A}X^a\\
     &= -\frac{g}{\rho^2}(D_X)_{BA}=\frac{g}{\rho^2}(D_X)_{AB}.
  \end{align*}
  The last equality is due to the fact that $D$ is a covariant
  derivative, which implies that
  $0=D_X\gb(N_A,N_B)=\gb(D_XN_A,N_B)+\gb(N_A,D_XN_B)$.
\end{proof}

\noindent Hence, it is possible to state Weingarten's formula entirely in terms of $\B_A$.

\begin{theorem}\label{thm:weingartenseq}
  Let $N_1,\ldots,N_p$ be an orthonormal basis of the normal
  space. For all $X\in\TSigma$ it holds that
  \begin{align}
    \frac{g}{\rho^2}\nablab_X N_A = -\B_A(X)+\sum_{B=1}^p\gb\paraa{\B_A(N_B),X}N_B,
  \end{align}  
  for $A=1,\ldots,p$.
\end{theorem}

\noindent Let us also note that since $\alpha(X,Y) = X^ag_{ac}(W_A)^c_bY^b$,
we can also rewrite Gauss' formula as
\begin{align}
  \nabla_X Y = \nablab_X Y-\frac{\rho^2}{g}\sum_{A=1}^p\gb\paraa{\B_A(X),Y}N_A.
\end{align}

\noindent In the particular case when $\rho=\sqrt{g}$ all formulas
simplify, and the most important ones become
\begin{align*}
  &K = \frac{1}{g}\gb\paraa{\Rb(e_1,e_2)e_2,e_1}
  -\frac{1}{2}\sum_{A=1}^p\Tr\SA^2,\\
  &H = \frac{1}{2}\sum_{A=1}^p\paraa{\Tr\B_A}N_A,\\
  &\nablab_X N_A = -\B_A(X)+\sum_{B=1}^p\gb\paraa{\B_A(N_B),X}N_B\\
  &\nabla_X Y = \nablab_X Y-\sum_{A=1}^p\gb\paraa{\B_A(X),Y}N_A.
\end{align*}

\noindent Note that if the ambient manifold $M$ is pseudo-Riemannian,
corresponding formulas can be worked out when the induced metric on
$\Sigma$ is non-degenerate.

\section{Complex structures and the construction of normal vectors}

\noindent To every Riemannian metric on $\Sigma$ one can associate an
almost complex structure $\J$ through the formula
\begin{equation*}
  \J(X) = \frac{1}{\sqrt{g}}\eps^{ac}g_{cb}X^be_a,
\end{equation*}
and since on a two dimensional manifold any almost complex structure
is integrable, $\J$ is a complex structure on $\Sigma$. It follows
immediately from Proposition \ref{prop:SPproperties} that this complex
structure can be expressed in terms of $\P$.

\begin{theorem}\label{thm:complexstructure}
  Defining $\JM(X)=(\rho/\sqrt{g})\P(X)$ for all $X\in TM$ yields the following results:
  \begin{enumerate}
  \item $\J_M(X)=\J(X)$ for all $X\in\TSigma$,
  \item $-\JM^2(X)=g^{ab}\gb(X,e_a)e_b$ is the projection of $X\in TM$ onto $\TSigma$. 
  \end{enumerate}
\end{theorem}

\noindent It is a standard result that the metric on $\Sigma$ is
hermitian with respect to the complex structure $\J$ and that the K\"ahler
form
\begin{align}
  \Omega(X,Y) = g\paraa{X,\J(Y)},
\end{align}
induces the Poisson bracket
\begin{align*}
  \{f,h\}_\Omega = \frac{1}{\sqrt{g}}\eps^{ab}\paraa{\d_af}\paraa{\d_bh}
  =\frac{\rho}{\sqrt{g}}\{f,h\}.
\end{align*}
Note that when choosing $\rho=\sqrt{g}$ (which implies $\JM=\P$),
the Poisson bracket induced from $\Omega$ coincides with the one
defined in (\ref{eq:PbracketDef}).

Theorem \ref{thm:complexstructure} also provides an expression for the
projection operator in terms of the Poisson bracket of the embedding
coordinates only, with no reference to the normal vectors. Therefore,
one can in principle construct $p$ orthonormal normal vectors from
$x^1,\ldots,x^m$ as follows: Choose an arbitrary frame
$Y_1,\ldots,Y_m$ of $TM$ and set
\begin{align}
  \Yt_k = Y_k + \JM^2(Y_k), 
\end{align}
which implies that $\Yt_k\in\TSigma^\perp$. Then apply the
Gram-Schmidt orthonormalization procedure to $\Yt_1,\ldots,\Yt_m$ to
obtain $p$ orthonormal vectors in $\TSigma^\perp$.

Another way of constructing normal vectors is obtained from the following result:

\begin{proposition}\label{prop:normalvectors}
  For any choice of $j_1,\ldots,j_{p-1}\in\{1,\ldots,m\}$ the vector
  \begin{align}
    Z_{j_1\cdots j_{p-1}}=\frac{\rho}{2\sqrt{g(p-1)!}}\gb^{ij}\eps_{jklj_1\cdots j_{p-1}}\{x^k,x^l\}\d_i,
  \end{align}
  where $\eps_{i_1\cdots i_m}$ is the Levi-Civita tensor, is normal to
  $\TSigma$, i.e. $\gb(Z_{j_1\cdots j_{p-1}},e_a)=0$ for $a=1,2$.
\end{proposition}

\begin{proof}
  The proof consists of a straightforward computation:
  \begin{align*}
    \frac{\sqrt{g(p-1)!}}{\rho}\gb(e_a,Z_{j_1\cdots j_{p-1}}) &= \frac{1}{2}\gb_{im}(\d_ax^m)\gb^{ij}\eps_{jklj_1\cdots j_{p-1}}\{x^k,x^l\}\\
    &= \frac{1}{2\rho}\eps_{jklj_1\cdots j_{p-1}}\eps^{bc}\paraa{\d_ax^j}\paraa{\d_bx^k}\paraa{\d_cx^l} = 0,
  \end{align*}
  since $a,b,c$ can only take on two different values and
  $\paraa{\d_ax^j}\paraa{\d_bx^k}\paraa{\d_cx^l}$ is contracted with
  $\eps_{jklj_1\cdots j_{p-1}}$, which is antisymmetric in $j$, $k$ and $l$.
\end{proof}

\noindent In general, $Z_{j_1\cdots j_{p-1}}$ defines $p(p+1)(p+2)/6$
normal vectors, and one can again apply the Gram-Schmidt
orthonormalization procedure to extract $p$ orthonormal vectors in
$\TSigma^\perp$. However, it turns out that one can use $Z_{j_1\cdots
  j_{p-1}}$ to construct another set of normal vectors, avoiding
explicit use of the Gram-Schmidt procedure, as follows: Let $I,J,L$
denote multi-indices of length $p-1$, i.e. $I=i_1\cdots i_{p-1}$, and
writing $Z_{i_1\cdots i_{p-1}} \equiv Z_I$ we introduce
\begin{align}
  \Z_I^J = \gb_{ij}Z_I^iZ^{Jj}.
\end{align}
Considered as a matrix over multi-indices, $\Z$ is symmetric and we
let $E_I,\mu_I$ denote orthonormal eigenvectors and their
corresponding eigenvalues. Using the eigenvectors to define
\begin{align}
  \Nh_I = E_I^JZ_J, 
\end{align}
one finds that $\gb(\Nh_I,\Nh_J) = \mu_I\delta_{IJ}$, i.e. the vectors are orthogonal.
\begin{lemma}\label{lemma:Zprojection}
  For $\Z_I^J$ defined as above, it holds that
  \begin{align}
    &\sum_L\Z_I^L\Z_L^J = \Z_I^J\label{eq:Zsquare}\\
    &\sum_I\Z_I^I = p.\label{eq:Ztrace}
  \end{align}
\end{lemma}

\begin{proof}
  Both (\ref{eq:Zsquare}) and (\ref{eq:Ztrace}) can easily be proven
  once one has the following result
  \begin{align}
    Z^i_KZ^{jK}
    = \gb^{ij}+\frac{\rho^2}{g}\paraa{\P^2}^{ij}.
  \end{align}
  The above formula is proven using that
  \begin{align*}
    \eps_{ijkK}\eps^{lmnK} &= (p-1)!\paraa{\delta_{[i}^{[l}\delta_{j}^m\delta_{k]}^{n]}}\\
    &=(p-1)!\bracketb{
      \delta_{i}^{l}\delta_{j}^{m}\delta_{k}^{n}-\delta_{i}^{l}\delta_{k}^{m}\delta_{j}^{n}
      -\delta_{j}^{l}\delta_{i}^{m}\delta_{k}^{n}+\delta_{j}^{l}\delta_{k}^{m}\delta_{i}^{n}
      -\delta_{k}^{l}\delta_{j}^{m}\delta_{i}^{n}+\delta_{k}^{l}\delta_{i}^{m}\delta_{j}^{n}
    },
  \end{align*}
  which gives 
  \begin{align*}
    Z^i_KZ^{jK} = \frac{\rho^2}{g}\bracketc{-\frac{1}{2}\paraa{\Tr\P^2}\gb^{ij}
    +\paraa{\P^2}^{ij}}
  =\gb^{ij}+\frac{\rho^2}{g}\paraa{\P^2}^{ij}.
  \end{align*}
  Formula (\ref{eq:Ztrace}) is now immediate, and to obtain
  (\ref{eq:Zsquare}) one simply notes that since $Z_J\in\TSigma^\perp$
  and $\P^2$ projects onto $\TSigma$, it holds that $\P^2(Z_J)=0$.
\end{proof}

\noindent From Lemma \ref{lemma:Zprojection} it follows that an
eigenvalue of $\Z$ is either 0 or 1, which implies that $\Nh_I=0$ or
$\gb(\Nh_I,\Nh_I)=1$, and that the number of non-zero vectors is
$\Tr\Z = \Z_I^I=p$. Hence, the $p$ non-zero vectors among $\Nh_I$
constitute an orthonormal basis of $\TSigma^\perp$, and it follows
that one can replace any sum over two normal vectors $N_A$ by a
contraction of $\Nh_I$ with $\Nh^I$. As an example, let us work out
explicit formulas for the Gaussian curvature and the mean curvature
vector in the case when $M=\reals^m$. For that, it is convenient to
have the following lemma at hand.

\begin{lemma}\label{lemma:Smovef}
  Defining
  \begin{align*}
    \S^{ij}(X) = \frac{1}{\rho}\eps^{ab}\paraa{\d_ax^i}\paraa{\nablab_bX}
  \end{align*}
  for $X\in TM$ it holds that
  \begin{align*}
    \Tr\S(fN)\S(hN')\equiv \S^i_j(fN)\S^j_i(hN') = fh\Tr\S(N)\S(N')
  \end{align*}
  for all $N,N'\in\TSigma^\perp$ and $f,h\in C^\infty(M)$. 
\end{lemma}

\begin{proof}
  Using $\nablab_bfN = (\d_bf)N+f\nablab_bN$ one obtains
  \begin{align*}
    \S^i_j(fN)\S^j_i(hN') &= fh\S^i_j(N)\S^j_i(N')\\
    &\quad+\frac{1}{\rho^2}\eps^{ab}\eps^{pq}\paraa{\d_ax^i}\paraa{\d_bf}N^j
    \gb_{jk}\paraa{\d_px^k}h\paraa{\nablab_qN'}^l\gb_{li}\\
    &\quad+\frac{1}{\rho^2}\eps^{ab}\eps^{pq}\paraa{\d_ax^i}f\paraa{\nablab_bN}^j
    \gb_{jk}\paraa{\d_px^k}\paraa{\d_qh}N'^l\gb_{li}\\
    &\quad+\frac{1}{\rho^2}\eps^{ab}\eps^{pq}\paraa{\d_ax^i}\paraa{\d_bf}N^j
    \gb_{jk}\paraa{\d_px^k}\paraa{\d_qh}N'^l\gb_{li}\\
    &=fh\S^i_j(N)\S^j_i(N')
  \end{align*}
  as $N^j\gb_{jk}\paraa{\d_px^k}=N'^l\gb_{li}\paraa{\d_ax^i}$=0 since $N,N'\in\TSigma^\perp$.
\end{proof}

\begin{proposition}
  If $M=\reals^m$ then
  \begin{align}
    K &= -\frac{\rho^4}{8g^2(p-1)!}
    \sum\eps_{jklI}\eps_{imnI}\{x^i,\{x^k,x^l\}\}\{x^j,\{x^m,x^n\}\}\\
    H &= \frac{\rho^4}{8g^2(p-1)!}
    \sum \eps_{iklI}\eps_{k'mnI}\{x^i,x^j\}\{x^j,\{x^k,x^l\}\}
    \{x^m,x^n\}\d_{k'}.
  \end{align}
\end{proposition}

\begin{proof}
  When $M=\reals^m$ we one use equation (\ref{eq:gausscurvRm}) to
  write
  \begin{align*}
    K = -\frac{\rho^2}{2g}\sum_{A=1}^p\sum_{i,j=1}^m\{x^i,n_A^j\}\{x^j,n_A^i\}
    =-\frac{\rho^2}{2g}\sum_I\sum_{i,j=1}^m\{x^i,\Nh_I^j\}\{x^j,\Nh_I^i\}
  \end{align*}
  since exactly $p$ of the vectors $\Nh_I$ are non-zero and
  orthonormal. Inserting the definition of $\Nh_I$ into the above
  expression and using Lemma \ref{lemma:Smovef}, together with the
  orthonormality of the eigenvectors $E_I$, yields
  \begin{align*}
    K &= -\frac{\rho^2}{2g}\sum_{I,J,L}\sum_{i,j=1}^m\{x^i,E_{I}^JZ_J^j\}\{x^j,E_I^LZ_L^i\}\\
    &=-\frac{\rho^2}{2g}\sum_{I,J,L}E_I^JE_I^L\sum_{i,j=1}^m\{x^i,Z_J^j\}\{x^j,Z_L^i\}\\
    & =-\frac{\rho^2}{2g}\sum_J\sum_{i,j=1}^m\{x^i,Z_J^j\}\{x^j,Z_J^i\}
  \end{align*}
  which, by inserting the definition of $Z_I$ and again using Lemma \ref{lemma:Smovef}, becomes
  \begin{equation*}
    K = -\frac{\rho^4}{8g^2(p-1)!}\sum\eps_{jklI}\eps_{imnI}\{x^i,\{x^j,x^k\}\}\{x^j,\{x^m,x^n\}\}.
  \end{equation*}
  The formula for $H$ can be proven analogously.
\end{proof}

\bibliographystyle{alpha}
\bibliography{psurface}

\end{document}